\documentclass[11pt]{amsart}
\usepackage[utf8]{inputenc}
\usepackage[text={6.1in,8.5in},centering]{geometry}
\usepackage{amssymb,amsmath,amsthm}
\usepackage{pgf, tikz}
\usetikzlibrary{shapes, patterns.meta}
\tikzset{
   smhex/.style={shape=regular polygon,regular polygon sides=6, minimum size=0.5cm, draw, inner sep=0},
    smbox/.style={shape=rectangle, minimum size=0.5cm, draw, inner sep=0},
   hexa/.style= {shape=regular polygon,regular polygon sides=6, minimum size=1cm, draw, inner sep=0}
}
\newcommand{\hexcol}[3]{
    \foreach \j in {1,...,#3}{
        \node[smhex] at ({(0.75 * #1},
                {(#2 + \j)*sin(60)}) {};}
}
\usepackage{caption}
\usepackage{subcaption}

\usepackage{wasysym} % just for comparison

\usetikzlibrary{patterns}
\usetikzlibrary{fadings}

\newtheorem{thm}[equation]{Theorem}

\newtheorem{lem}[equation]{Lemma}
\newtheorem{prop}[equation]{Proposition}
\newtheorem{cor}[equation]{Corollary}

\theoremstyle{definition}
\newtheorem{defn}{Definition}
\newtheorem{rmk}[equation]{Remark}

\newtheorem*{thm*}{Theorem}
\newtheorem*{cor*}{Corollary}
\newtheorem*{lem*}{Lemma}

\numberwithin{equation}{section}

\title{Parity Property of Hexagonal Sliding Puzzles}
\author[R.~Karpman]{Ray Karpman}
\author[E.~Rold\'an]{\'Erika Rold\'an}


\begin{document}

\maketitle

\begin{abstract}
We study the puzzle graphs of hexagonal sliding puzzles of various shapes, and with various numbers of holes. The puzzle graph is a combinatorial model which captures the solvability and the complexity of sequential mechanical puzzles. Questions relating to the puzzle graph have been previously studied and resolved for the 15 Puzzle which is the most famous, and unsolvable, square sliding puzzle of all time.  It is known that for square puzzles such as the 15 Puzzle, solvability depends on a parity property that splits the puzzle graph into two components. In the case of hexagonal sliding puzzles we get more interesting parity properties that depend on the shape of the boards and on the missing tiles or holes on the board. We show that for large-enough hexagonal, triangular, or parallelogram-shaped boards with hexagonal tiles, all puzzles with three or more holes are solvable. For puzzles with two or more holes, we give a solvability criterion involving both a parity property, and the placement of tiles in \emph{tight corners} of the board. The puzzle graph is a discrete model for the configuration space of hard tiles (hexagons or squares) moving on different tessellation based domains. Understanding the combinatorics of the puzzle graph could lead to understanding some aspects of the topology of these configuration spaces. 
\end{abstract}

\section{Introduction}

Sliding puzzles are sequential mechanical puzzles which are solved by sliding certain pieces on a fixed board from a starting configuration to a target final configuration. We focus here on sliding puzzles whose boards consist of a finite subset of tiles of the square or hexagonal regular tessellations of the plane. If the board is completely covered by tiles, squares or hexagons depending on the selected tessellation, then the tiles can't move at all. In this case, any sliding puzzle defined on that board is clearly unsolvable. If some tiles of the board are removed, then holes are created and it becomes possible to slide some of the tiles; as a consequence, some sliding puzzles become solvable. Here, we are interested in determining which sliding puzzles on a given board are solvable, for boards of various shapes and with various number of holes. This problem has previously been solved for a large family of square sliding puzzles \cite{johnson1879notes, story1879notes, wilson1974graph}. However, the answers for boards with hexagonal tiles were previously unknown.

%Sliding puzzles can be seen as discrete models of disk configuration spaces\cite{alpert2020configuration,  alpert2021configuration, alpert2020discrete}.nalyzing the parity properties and maximal connectivity of a sliding puzzles corresponds to analyzing the topology of discrete configuration spaces generated by Other examples of discrete configuration spaces are the ones use to model robot motion within a certain domain that us usually represented by a graph.  or such as robots moving in a factory unwilling to make collision \cite{abrams2002finding}, within a given topological space. Moreover, configuration spaces can also be used to study certain properties arising from mathematical \textit{objects} such as topological and geometric properties of \cite{westerland2012configuration}

The majority of sliding puzzles that have been physically created and commercialized have rectangular shaped boards and square tiles. However, there also exist more general \textit{sliding block puzzles} with the pieces unlabeled, or with the pieces consisting of other very simple small polyominoes such as dominoes, triominoes, or tetrominoes. Edward Hordern had the world-wide biggest collection of sliding block puzzles that is now part of \textit{the Puzzle Museum} collection \cite{puzzle}. In 1986 \cite{hordern1986sliding}, Hordern wrote a wonderful book describing more than 250 sliding block puzzles of his collection.
%In this book, he also included one solution for each one of the puzzles (the best known at that time with respect to the number of movements needed to be solved). 

With very few exceptions, for sliding puzzles with labeled unit square tiles, if there is exactly one removed tile from the board, then a parity property determines if a puzzle is solvable or not; with two or more tiles removed from the board, all sliding puzzles defined on the board are solvable. This was proved for the first time during the decade of the 1880's \cite{johnson1879notes, story1879notes}, when mathematicians and the general public got attracted, obsessed, and even some times traumatized, by the most famous (unsolvable) sliding puzzle of all times: The 15 Puzzle \cite{slocum200615}. 

The 15 Puzzle consists of a $4 \times 4$ square board divided into 16 unit squared with 15 square tiles labeled from 1 to 15 placed on the board, and with one empty square. The initial configuration of the tiles in the 15 Puzzle has the tiles 1 to 13 placed in ascendant order (from left to right and top to bottom), and the tiles 14 and 15 interchanged. The target configuration of this puzzle leaves the first 13 ordered tiles fixed and interchanges the 14 and 15 tiles---see Figure \ref{classicpuz}.

\begin{figure}
    \centering
    \begin{tikzpicture}
    \node[smbox] at (0,0) {};
    \foreach \j in {1,...,3}{
        \node[smbox] at ({0.5*\j}, 0) {\j};
    }
    \foreach \j in {4,...,7}{
        \node[smbox] at ({0.5*(\j-4)}, -0.5) {\j};
    }
    \foreach \j in {8,...,11}{
        \node[smbox] at ({0.5*(\j-8)}, -1) {\j};
    }
    \foreach \j in {12,13}{
        \node[smbox] at ({0.5*(\j-12)}, -1.5) {\j};
    }
    \foreach \j in {14,15}{
        \node[preaction={fill=gray!30}, smbox] at ({0.5*(\j-12)}, -1.5) {\j};
    }
    \begin{scope}[xshift=3 cm]
        \node[smbox] at (0,0) {};
    \foreach \j in {1,...,3}{
        \node[smbox] at ({0.5*\j}, 0) {\j};
    }
    \foreach \j in {4,...,7}{
        \node[smbox] at ({0.5*(\j-4)}, -0.5) {\j};
    }
    \foreach \j in {8,...,11}{
        \node[smbox] at ({0.5*(\j-8)}, -1) {\j};
    }
    \foreach \j in {12,13}{
        \node[smbox] at ({0.5*(\j-12)}, -1.5) {\j};
    }
    \foreach \j in {14,15}{
        \pgfmathsetmacro\tile{int(29-\j)}
        \node[preaction={fill=gray!30}, smbox] at ({0.5*(\j-12)}, -1.5) {\tile};
    }
    \end{scope}
    \end{tikzpicture}
    \caption{The most famous sliding puzzle is the 15 Puzzle that with our notation is represented by $R_{\square}(1;4)$. It is not possible to transform the puzzle at right to the one at left by sliding the tiles.} %% Also define triangular and 'flower' boards :)}
    \label{classicpuz}
\end{figure}

In 1879 \cite{johnson1879notes}, W. Johnson proved, using a parity argument, that the 15 Puzzle has no solution. He showed that if one starts from a specific configuration, it is impossible to reach any other configuration that is obtained by applying an \emph{odd permutation} to the labels of the tiles. Recall that a permutation is \emph{even} if it can be written as a product of an even number of transpositions--that is, of permutations which simply interchange two labels. A permutation is \emph{odd} if it can be written as a product of an odd number of transpositions. Since every permutation is a product of transpositions, every permutation is either even or odd.

Johnson's proof works board in the shape of an $m_1 \times m_2$ rectangle with $m_1,m_2 \geq 2$ that has exactly one missing tile and the rest of the tiles being unit squares labeled with consecutive integers. We say that a board (with the respective tiles removed) has the \emph{weak parity property} if we cannot reach any configuration corresponding to an odd permutation of the tiles by performing permitted slides. Hence any rectangular board with one tile removed has the weak parity property.

Complementing Johnson's analysis, W.E. Story proved in the second part of the same paper  \cite{story1879notes}, that starting from a fixed configuration \textit{all} configurations that can be obtained by an even permutation of the tiles will lead to a puzzle that has a solution. We say a  board (with the respective tiles removed) has the \emph{strong parity property} if we are able to obtain all configurations that require an even number of transpositions of tiles from a fixed starting configuration.

Putting together the results by Johnson and Story tells us that it is possible to solve a sliding puzzle on a rectangular shaped board with square tiles and exactly one tile missing, if and only if the following holds: if we slide tiles of the starting configuration so that the missing tile is in the same position as in the target, the permutations of tiles in the starting and the final configurations have the same parity. In other words, Johnson and Story proved that the set of all configurations is partitioned in two big sets that are determined by parity and that this partition characterizes when a puzzle, a given started and final configuration, has a solution or not. A further natural question then arises: how many more tiles need to be removed from board for all sliding puzzles defined on it to be solvable. Trivially, for a rectangular shaped board with unit square tiles it suffices to remove two tiles. We call a board with the respective tiles removed needed to have this property to be \textit{maximally connected}.

%Combining the results of Johnson and Story, we see that whether a sliding puzzle on a square board is solvable or not can be completely characterized in terms of the parity of a permutation.

% While Johnson and Story considered only rectangular boards with a single tile missing, we vary the shape of the board and the number of missing tiles, which yields several families of puzzles.
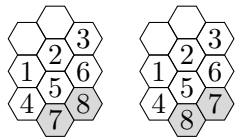
\begin{figure}
    \centering
    \begin{tikzpicture}
    [scale=0.5]
    \foreach \i in {0,1}{
        \node[smhex] at (\i*0.75,{\i*0.5*sin(60)}) {};
    }
    \foreach \i in {1,...,3}{
        \node[smhex] at ({(\i-1)*0.75}, {((\i-3)*0.5*sin(60)}) {\i};
    }
    \foreach \i in {4,...,6}{
        \node[smhex] at ({(\i-4)*0.75}, {((\i-8)*0.5*sin(60)}) {\i};
    }
    \node[preaction={fill=gray!30}, smhex] at (0.75, {-2.5*sin(60)}) {7};
     \node[preaction={fill=gray!30}, smhex] at (1.5, {-2*sin(60)}) {8};
   \begin{scope}[xshift=3.5cm]
    \foreach \i in {0,1}{
        \node[smhex] at (\i*0.75,{\i*0.5*sin(60)}) {};
    }
    \foreach \i in {1,...,3}{
        \node[smhex] at ({(\i-1)*0.75}, {((\i-3)*0.5*sin(60)}) {\i};
    }
    \foreach \i in {4,...,6}{
        \node[smhex] at ({(\i-4)*0.75}, {((\i-8)*0.5*sin(60)}) {\i};
    }
    \node[preaction={fill=gray!30}, smhex] at (0.75, {-2.5*sin(60)}) {8};
    \node[preaction={fill=gray!30}, smhex] at (1.5, {-2*sin(60)}) {7};
    \end{scope}
    \end{tikzpicture}
    \caption{We prove that is not possible to transform the puzzle at right into the puzzle at left by sliding tiles.}
    \label{trimmedpuz}
\end{figure}

In this paper, we analyze the weak and strong parity properties, and maximally connectivity of boards with different shapes with hexagonal tiles---see Figure \ref{boards} for some examples of these boards. Of the families of hexagonal boards that we study here, only one family, parallelogram shaped boards with hexagonal tiles, have previously been define and studied by H. Alpert \cite{alpert2020discrete}. In her paper, Alpert analyzes, with respect of the size of the boards, asymptotically how fast these parallelogram shaped hexagonal sliding puzzles can be solved. For this, she first proved that these family of parallelogram boards are maximally connected, for sizes bigger than $5 \times 5$,  with 6 or more hexagonal tiles remove and at least two of the removed hexagonal tiles sharing an edge. Thus, Alpert studies only maximally connected boards with hexagonal tiles. With respect to this particular family of boards, we prove here that maximal connectivity is reached with 3 or more hexagonal tiles removed (even for small sized parallelogram boards), and that they have the weak parity property (but not the strong parity property) with two tiles removed. Thus, we show that, unlike the 15 Puzzle or more general square, the parallelogram sliding puzzles never have the strong parity property. 

 Given a specific sliding puzzle on a board, we can have a sense on how difficult it will be to solve it by knowing the minimum number of single slides required to get from the starting to the targeted final configuration of the puzzle. This defines a distance on the set of configurations and a natural way of measuring the complexity of a sliding puzzle. With this perspective, the \textit{most difficult puzzles} that can be defined on a given board are those that require more slides to be solved. This is captured by the \textit{God's number} which is the maximum of the distances between configurations. 

In the specific case of the board of the 15 Puzzle, its God's number has been computed and it is equal to 80 \cite{brungger1999parallel}. A well known open problem is to determine God's number for bigger square shaped (or rectangular shaped) boards. The problem of finding the God's number of squared shaped boards with labeled square tiles is NP-Hard \cite{ratner1986finding}. Here, we also give bounds for the God's number of some hexagonal sliding puzzles. 
In the rest of this section we give precise statements of our main results.  
\subsection{Main Results}
    %The first ingredient for a sliding puzzle is a \textit{board}, which consists of any subset of tiles of a tessellation that has a connected interior. In the literature, these are also known as polyforms. Here, we mainly study sliding puzzles on the square and hexagonal regular tessellations. The tiles on the board can be distinguishable, usually by labeling each tile with a different number, or indistinguishable. If a tile is \textit{not on the board} it generates a hole that could allow the pieces to slide to occupy this hole.
\begin{defn}
    Let $B(h)$ denote the board $B$ with $h$ holes, that is, with $h$ unoccupied positions of the board. We define the puzzle graph of $B(h)$, which we denote as $puz[B(h)]$, as the graph that contains as vertices each one of the possible placements (or configurations) of the labeled tiles that are on $B$ and that has an edge between two vertices whenever it is possible to go from one configuration to the other by sliding one tile into a hole.
\end{defn}

If two tiles are removed and share an edge, topologically both together would count as only one hole. Nevertheless, we count each unoccupied position of the board as a distinct hole. When the number of holes does not need to be specified we will abuse notation and denote the board with any number of holes simply by $B$.

\begin{defn}
We say a configuration of a board $B$ is \emph{isolated} if no slides are possible. A configuration is \emph{non-isolated} if there is at least one tile that can slide.
\end{defn}

For a board with squares tiles, isolated configurations are only possible for boards with no holes, as it is always possible to slide a square tile into a hole that shares an edge with the tile. For a board whose tiles are hexagons, a tile can only be slide to a neighboring empty hexagonal hole if and only if there are two adjacent empty hexagons that are at the same time sharing an edge with the tile to move---see Figure \ref{fig:my_label}.

\begin{figure}
    \centering
    \begin{tikzpicture}
    \node[hexa] at (0,0) {$a$};
    \node[preaction={fill=gray!30}, hexa] at (0.75, {sin(60)*0.5}) {};
    \node[preaction={fill=gray!30}, hexa] at (0.75, {-sin(60)*0.5}) {};
    \begin{scope}[xshift=2.75cm]
    \node[preaction={fill=gray!30}, hexa] at (0,0) {};
    \node[preaction={fill=gray!30}, hexa] at (0.75, {sin(60)*0.5}) {};
    \node[preaction={fill=gray!30}, hexa] at (0.75, {-sin(60)*0.5}) {};
    \node[preaction={fill=white}, hexa] at (0.5,0) {$a$};
    \end{scope}
    \begin{scope}[xshift=5.5cm]
    \node[preaction={fill=gray!30}, hexa] at (0,0) {};
    \node[hexa] at (0.75, {sin(60)*0.5}) {$a$};
    \node[preaction={fill=gray!30}, hexa] at (0.75, {-sin(60)*0.5}) {};
    \end{scope}
    \end{tikzpicture}
    \caption{As noticed in \cite{alpert2020discrete}, an hexagonal tile which is adjacent to a pair of holes may slide into either neighboring hole.}
    \label{fig:my_label}
\end{figure}
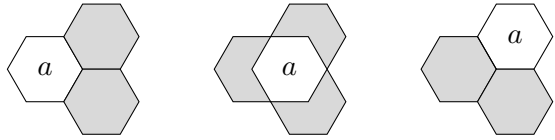

\begin{defn}
    We say that $puz[B(h)]$ has the \emph{weak parity property} if whenever two non-isolated configurations with the holes in the same positions are in the same connected component, then one can be obtained from the other by an even permutation of the tiles.
\end{defn}
\begin{defn}
We say that $puz[B(h)]$ has the \emph{strong parity property} if it has the weak parity property and its puzzle graph has exactly two components containing all (non-isolated) configurations.
\end{defn}
Hence $puz[B(h)]$ has the strong parity property if two non-isolated configurations of $B$ are in the same connected component of the puzzle graph exactly when one can be obtained from the other by an even permutation of the tiles.

%%%
\begin{defn}
We say the puzzle graph of $B(h)$ is \emph{maximally connected} if it has one large connected component containing all non-isolated configurations. We call the minimum number $h$ such that the puzzle graph of $B(h)$ is maximally connected the \emph{connection number} of $puz[B(h)]$.
\end{defn}

 We will denote the $m_1 \times m_2$ square grid board as $R_{\square}(h;m_1,m_2)$, were $h$ represents the number of holes. In the particular case when $m=m_1=m_2$ we simplify the notation by $R_{\square}(h;m)$. 

We will denote the $m_1 \times m_2$ parallelogram hexagonal board as $P_{\varhexagon}(h;m_1,m_2)$, were $h$ represents the number of holes. In the particular case when $m=m_1=m_2$ we simplify the notation by $P_{\varhexagon}(h;m)$. 

We investigate parallelogram boards of all sizes, and with varying numbers of holes. We note that the $1 \times m_2$ square board $P_{\square}(1:1,m_2)$ has $(m_2-1)!$ connected components, one for each possible permutation of $m_2-1$ tiles, as we can only slide the tiles back and forth without changing their order. It is natural to ask if there exists an hexagonal board shape for hexagonal sliding puzzles that exhibits analogous behaviour. We find that this is the case for $P_{\varhexagon}(h;2,m_2)$ boards.

\begin{figure}
    \centering
    \begin{tikzpicture}
    [scale=0.5]
    \foreach \i in {1,...,2}{
        \pgfmathsetmacro{\start}{-1-0.5*\i}
        \hexcol{\i}{\start}{3}
    }
    \begin{scope}[xshift=2.75cm]
    \foreach \i in {1,...,3}{
        \pgfmathsetmacro{\start}{-1-0.5*\i}
        \hexcol{\i}{\start}{4}
    }
    \end{scope}
    \begin{scope}[xshift=8.25cm]
    \foreach \i in {1,...,3}{
        \hexcol{\i}{-0.5 * \i}{\i}
    }
    \end{scope}
    \begin{scope}[xshift=11.75cm]
    \foreach \i in {1,...,4}{
        \hexcol{\i}{-0.5 * \i}{\i}
    }
    \end{scope}
    \begin{scope}[xshift=18cm]
        \hexcol{1}{-1}{2}
        \hexcol{2}{-1.5}{3}
        \hexcol{3}{-1}{2}
    \end{scope}
    \begin{scope}[xshift=21.5cm]
    \foreach \i in {1,...,3}{
        \pgfmathsetmacro{\start}{-1-0.5*\i}
        \pgfmathsetmacro{\stop}{\i+2}
        \hexcol{\i}{\start}{\stop}
    }
    \foreach \i in {1,2}{
        \pgfmathsetmacro{\colnum}{\i+3}
        \pgfmathsetmacro{\start}{0.5*\i-2.5}
        \pgfmathsetmacro{\stop}{5-\i}
        \hexcol{\colnum}{\start}{\stop}
    }
    \end{scope}
    \end{tikzpicture}
    \caption{Left to right: A $2 \times 3$ parallelogram, a $3 \times 4$ parallelogram, a $3 \times 3$ triangle, a $4 \times 4$ triangle, a flower with 2 layers, and a flower with 3 layers.}
    \label{boards}
\end{figure}
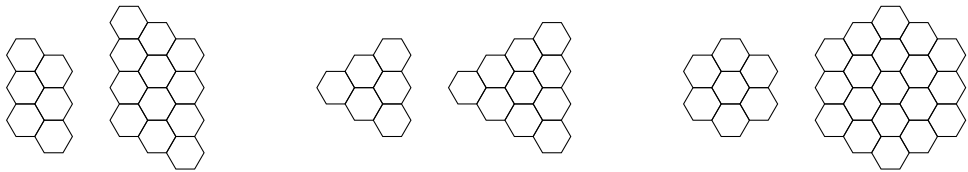

\begin{thm}\label{thm:skiny}
For $m_1 =2$ and $m_2 \geq 2$, $puz[P_{\varhexagon}(h; m_1,m_2)]$ has $(m_1m_2-h)!$ connected component containing non-isolated configurations. 
\end{thm}

 For a rectangular shape board in the square case when one has two holes the puzzle graph gets maximally connected. This is not the case for hexagonal sliding puzzles. For instance, we find that the $3 \times 3$ parallelogram with two holes exhibits a parity property similar to what is seen for the 15 puzzle. If we remove three holes instead of two, the resulting puzzle graph is maximally connected, meaning that every sliding puzzle on a $3 \times 3$ board with three holes is solvable.
%% Add figure!!

As mentioned before, the only family of hexagonal boards that has been introduced and studied before are large parallelogram boards $P_{\varhexagon}(h;m)$ by Alpert in 2020 \cite{alpert2020discrete}. 
\begin{prop}[Alpert \cite{alpert2020discrete}]
For $m \geq 5$ and $h \geq 6$, $puz[P_{\varhexagon}(h;m)]$ is maximally connected.
\end{prop}

Alpert's result states that if one has at least 6 removed tiles from a parallelogram hexagonal board that is at least of size $5\times5$, then the puzzle graph of the board gets as connected as possible. Our next theorem improves on Alpert's result by showing that we can reach maximal connectivity with only three holes. In fact, we shall see below that this is the exact number of tiles which need to be removed to reach maximal connectivity.

\begin{thm}\label{thm:parallelogram}
For $m_1 \geq 3$ and $m_2 \geq 3$, the puzzle graph $puz[P_{\varhexagon}(3; m_1,m_2)]$ is maximally connected. 
\end{thm}

Using this theorem, and an inductive approach involving \emph{patching} together smaller boards, we explore parity properties, maximal connectivity and the connection number of puzzle graphs of families of boards of other shapes. We now define some notation to use for these specific families of boards with hexagonal tiles.

Recall that we write $P_{\varhexagon}(h;m_1,m_2)$ to denote a board in the shape of an $m_1 \times m_2$ parallelogram with $h$ holes. Here $m_1$ denotes the number of columns, and $m_2$ the number of rows. When $m_1=m_2$, we simplify this notation to $P_{\varhexagon}(h;m)$. When we wish to refer only to the shape of the board, not the number of holes, we drop the $h$ and write $P_{\varhexagon}(m_1,m_2)$.

Similarly, we write $T_{\varhexagon}(h;m)$ to denote a board in the shape of an equilateral triangle, with $m$ tiles on a side. We write $F_{\varhexagon}(h;m)$ to denote a board in the shape of a hexagon, again with $m$ tiles on a side. To avoid confusion with the hexagonal shape of an individual tile, we refer to these hexagon-shaped boards as ''flowers'' in the text. We may think of a flower as being constructed by concentric rings of tiles around a central tile, hence $F_{\varhexagon}(h;m)$ may be referred to as a flower with $m$ layers.

\begin{cor}
    \label{cor:patchedboards}
    For $h \geq 3$, the puzzle graph of $puz[T_{\varhexagon}(h:m)]$ is maximally connected for $m \geq 5$, and the puzzle graph of $puz[F_{\varhexagon}(h:m)]$ is maximally connected for $m \geq 3$.
\end{cor}

We now investigate boards with exactly two holes missing. Here, as in the case of rectangular boards with one hole, parity is a key consideration. 

\begin{thm}\label{paritytheorem}
 The board $P_{\varhexagon}(2; m_1,m_2)$ has the weak parity property for $m_1, m_2 \geq 3.$
%    Let $C$ and $C'$ be two non-isolated configurations in the same connected component of $hex(m_1m_2-2:m_1, m_2)$, where the two holes are located in the same position in $C$ and $C'$. Then the permutation of the tiles necessary to transform $C$ into $C'$ is even.
\end{thm}

\begin{cor}\label{cor:weakparity}
    Any board with exactly two holes that is a sub-board of a parallelogram board has the weak parity property. Hence boards of any shape that we explore in this paper with exactly two holes have the weak parity property.
\end{cor}

\begin{thm}\label{thm:connectionnumber}
    The connection number is $3$ for the following boards: $P_{\varhexagon}(m_1,m_2)$ with $m_1,m_2 \geq 3$, $T_{\varhexagon}(m)$ with $m \geq 5$, and $F_{\varhexagon}(m)$ for $m \geq 3.$
\end{thm}

%Talk about the results of the square sliding puzzle. 

%[Example, patch 2] Trimmed $4 \times 4$ parallelogram (hopefully the even/odd configurations).

%We will have a paragraph describing the board patching techniques and then we state theorems about the families that we want to emphasise that can be build with this patching techniques. Refer to Hannah's paper because there is a lot of patching in her paper. 

It might be natural to hope that e.g. parallelogram-shaped boards have the strong parity property as well. However, this is not the case, essentially because tiles get ``stuck'' in corners of the board. To get around this problem, we define a \emph{tight corner} of a board to be any tile with exactly two neighbors. We define \emph{trimmed board} to be the result of removing all tight corners from an initial board. In principle, removing tight corners from a board could create new corners, by removing neighbors of the tiles left behind. Fortunately, if we start with a large-enough triangle or parallelogram, the resulting trimmed board has no tight corners. See Figure \ref{trimmedboards}.

We call the result of trimming an $m_1 \times m_2$ parallelogram board a \emph{trimmed $m_1 \times m_2$ parallelogram}, which we denote $P_{\varhexagon}^{tr}(m_1,m_2)$. Similarly, we write $T_{\hexagon}^{tr}(m)$ for the result of trimming a triangular board with $m$ hexagons on a side. See Figure \ref{trimmedboards}.

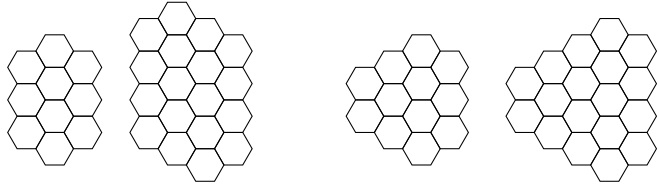
\begin{figure}
    \centering
    \begin{tikzpicture}
    [scale=0.5]
    \hexcol{1}{-1.5}{3}
    \hexcol{2}{-2}{4}
    \hexcol{3}{-1.5}{3}
    
    \begin{scope}[xshift=3.25cm]
    \hexcol{1}{-1.5}{4}
    \foreach \i in {2,...,3}{
        \pgfmathsetmacro{\start}{-1-0.5*\i}
        \hexcol{\i}{\start}{5}
    }
    \hexcol{4}{-2}{4}
    \end{scope}
    \begin{scope}[xshift=9cm]
    \foreach \i in {2,...,4}{
        \pgfmathsetmacro{\rownum}{\i-1}
        \hexcol{\rownum}{-0.5 * \i}{\i}
    }
    \hexcol{4}{-1.5}{3}
    \end{scope}
    \begin{scope}[xshift=13.25cm]
    \foreach \i in {2,...,5}{
        \pgfmathsetmacro{\rownum}{\i-1}
        \hexcol{\rownum}{-0.5 * \i}{\i}
    }
    \hexcol{5}{-2}{4}
    \end{scope}
    \end{tikzpicture}
    \caption{Left to right: the trimmed parallelograms $P_{\varhexagon}^{tr}(3,4)$ and $P_{\varhexagon}^{tr}(4,5)$, and the trimmed triangles $P_{\varhexagon}^{tr}(5)$ and $P_{\varhexagon}^{tr}(6)$.}
    \label{trimmedboards}
\end{figure}
%Perhaps...If a board has no 2 x n parts or acute angles then:
\begin{thm}
\label{thm:strongparity}
The puzzle graphs of the following boards with two holes removed have the strong parity property: 
\begin{enumerate}
    \item $ F_{\varhexagon}(m)$ for $m \geq 3$
    \item $T_{\varhexagon}^{tr}(m)$ for $m \geq 5$ 
    \item $P_{\varhexagon}^{tr}(m_1,m_2)$ where $m_1,m_2 \geq 3$ and $\max\{m_1,m_2\}\geq 4$.
    \end{enumerate}
\end{thm}

Hence flowers, trimmed parallelograms, and trimmed triangles give an intriguing new family of sliding puzzles which exhibit similar behavior to the classic 15 puzzle. For parallelogram and triangular boards with two holes, the story is a bit more complicated. Tiles can get `stuck'' in corners, a behavior that is not seen for square boards. As a result, we obtain many connected components containing non-isolated configurations, as detailed in the following theorem.
\begin{thm}
    \label{thm:countcomponents}
    If the board $B$ is
    \begin{enumerate}
        \item A parallelogram $P_{\varhexagon}(m_1,m_2)$, where $m_1,m_2\geq 3$ and $\max\{m_1,m_2\} \geq 4.$
        \item A triangle
        $T_{\varhexagon}(m)$ where $m \geq 5$.
    \end{enumerate}
    Then the number of connected components of $puz[B(2)]$ containing non-isolated configurations is given by
    \begin{itemize}
        \item $\displaystyle 4  \binom{m_1m_2-2}{2}$ for a parallelogram,
        \item $\displaystyle 12 \binom{m(m+1)/2-2}{3}$ for a triangle.
    \end{itemize}
\end{thm}

We calculate these God's numbers using the  Breath First Search algorithm that we have run on a specific component of a non-isolated configuration. By symmetry, all components of non-isolated configurations will have the same God's number.

The rest of the paper is structured as follows. In Section \ref{section:parallelogram}, we give a proof of Theorem \ref{thm:skiny} and Theorem \ref{thm:parallelogram}. In Section \ref{section:patching}, we prove a general result that allows us to build boards by gluing smaller boards in such a way that we preserve the connectivity of the puzzle graph. This allows us to prove Corollary \ref{cor:patchedboards}, which states that the puzzle graph is maximally connected for large-enough triangular and flower boards with three holes. 

Section \ref{section:parity} investigates boards with exactly two holes. Here, parity properties come into play. The section begins with a discussion of \emph{tight corners} of parallelogram and triangular boards, which can be an obstruction to connectivity for the puzzle graph. We then prove Theorem \ref{paritytheorem}, which shows that a large class of boards have the \emph{weak} parity property. This establishes Theorem \ref{thm:connectionnumber}, which gives the connection number for large-enough parallelogram, triangle, and flower-shaped boards. We adapt our patching method to show that many of the boards in question also have the \emph{strong} parity property,  puzzles with similar behavior to square puzzles. See Theorem 1.8. We end this with Theorem \ref{thm:countcomponents}, which gives the number of connected components containing non-isolated configurations in the puzzle graphs of triangular and puzzle-shaped boards with only two holes. Note that for these boards, the existence of tight corners leads to a large number of components. 

Throughout the paper, we make use of inductive arguments, with small boards of a given shape and given number of holes servings as the base case for the induction. We find the God's number and the number of connected components of the puzzle graphs of these smaller boards computationally, using Python code available in our GitHub repository \cite{hexagonSoftware}. Summaries of our computational results, with additional details, are found in Section \ref{section:computations}.

%% Note that we have computational results for smaller boards in last section--needs to be added!

%\begin{thm}
%For $m_1 \geq 3$ and $m_2 \geq 3$, if we remove the 2 tiles on the acute corners of the parallelogram, then if $n\leq (m_1 \times m_2) - 2$ then there is a large connected component on $Conf[hex(n; m_1,m_2)]$ containing every configuration for which not all the holes are isolated. 
%\end{thm}

\section{First analysis on parallelogram-shaped hex boards}
\label{section:parallelogram}

We start by proving Theorem \ref{thm:skiny} that says that for $m_1 =2$ and $m_2 \geq 2$, $puz[P_{\varhexagon}(h; m_1,m_2)]$ has $(m_1m_2-h)!$ connected components containing non-isolated configurations. 

\begin{proof}
We orient our $2 \times m_2$ board so that the hexagonal tiles form columns of length $m_2$ in the vertical direction, with the center of the topmost tile in the left column slightly higher than the center of the topmost tile in the right column. See Figure \ref{2bym}.

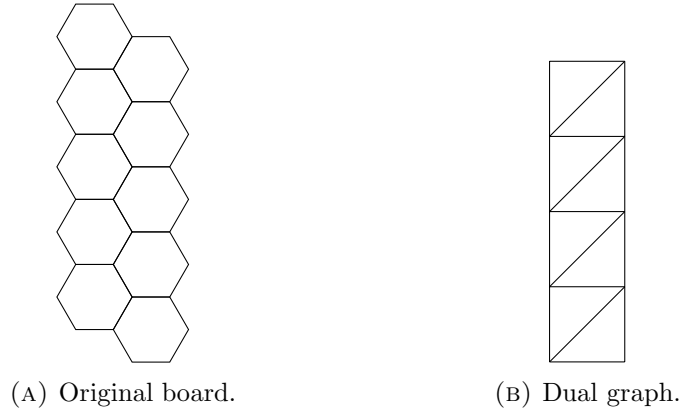
\begin{figure}
\begin{subfigure}[b]{0.3 \textwidth}
\centering
\begin{tikzpicture} 
[hexa/.style= {shape=regular polygon,regular polygon sides=6, minimum size=1cm, draw,inner sep=0,anchor=south}]
\foreach \j in {0,...,4}{
    \foreach \i in {0,1}{
        \node[hexa] (h\i;\j) at ({(0.75 * \i},{(\j-\i/2)*sin(60)}) {};}  } 
\end{tikzpicture}
\caption{Original board.}
\end{subfigure}
\hspace{0.5 in}
\begin{subfigure}[b]{0.3 \textwidth}
\centering
\begin{tikzpicture} 
\draw(0,0) grid (1, 4);
\foreach \j in {0,...,3}{
    \draw (0, \j) -- (1, {\j+1});}
\end{tikzpicture}
\caption{Dual graph.}
\end{subfigure}
\caption{A $2 \times m_2$ hex board with $m_2 = 5$, and its dual graph.}
\label{2bym}
\end{figure}

We may view the board as a graph, with vertices and edges corresponding to the corners and sides of hexagonal tiles or holes. It is convenient to consider the planar dual of this graph. To construct the dual, we draw a vertex in the center of each hexagon of the board. Two vertices in the \textit{dual graph} are adjacent precisely when the corresponding hexagons share an edge. We may then model the movement of tiles and holes on the original board by assigning tile labels to the corresponding vertices of the dual graph, and letting holes correspond to unlabeled vertices. We may slide a label to an unlabeled vertex precisely when the labeled vertex forms a triangle with two vertices that are unlabeled.

With these conventions, the dual graph of a $2 \times m_2$ board can be easily deformed to a $2 \times m_2$ grid with rows of length $2$ and columns of length $m_2$, and with diagonal edges connecting the bottom-left and top-right of each square. For convenience, we number the rows from top to bottom, starting with row $1$.

A configuration of a $2 \times m_2$ board is non-isolated precisely when there are at least two holes in adjacent positions. Let $n$ denote the number of tiles. If $n$ is even, so that $n = (2m_2-h)= 2k$ for some integer $k$, we can always perform slides on a non-isolated configuration such that the top $k$ rows are entirely filled with tiles, and the bottom $m_2 - k$ rows are entirely filled with holes. Similarly if $n = (2m_2-h)= 2k+1$ for some integer $k$, we can move the tiles so that the top $k$ rows are filled with tiles, and there is a tile in the left space of row $k+1$. In either case, we will say a board with tiles and holes positioned in this way is a \emph{home configuration}. 

Notice that home configurations are by construction non-isolated as long as $n \leq 2m_2 - 2$, that is equivalent to $h \geq 2$. Moreover, every component of $puz[P_{\varhexagon}(h; 2,m_2)]$ consisting of non-isolated configurations must contain at least one home configuration. The number of home configurations is $n!$, the number of possible permutations of the tiles (keeping the holes fixed). Hence to prove the proposition, it suffices to show that no two home configurations are in the same connected component of the puzzle graph.

Consider two labels in the dual graph corresponding to a $2 \times m_2$ board, say $a$ and $b$. We say $a$ is $\emph{weakly above}$ $b$ if either the row containing $a$ is above the row containing $b$, or $a$ and $b$ are in the same row with $a$ on the left. Notice that in a graph corresponding to a home configuration, the label at left in the $t^{th}$ row from the top is the unique label with exactly $2(t-1)$ labels weakly above it. The label at right in the $t^{th}$ row is the unique label with $2(t-1)+1$ labels weakly above it. 

Consider two home configurations $C_1, C_2$ which are in the same connected component of the puzzle graph, and let $j$ be any tile label which appears in $C_1$ and $C_2$. We claim that the set of labels which are weakly above $j$ is the same in the graphs of $C_1$ and $C_2$. In particular, the number of labels which appear weakly above $j$ is the same in the graphs of $C_1$ and $C_2$. But this, in turn, shows that $j$ must appear in the same position in both configurations, by the argument of the previous paragraph. Hence all tiles in $C_1$ and $C_2$ appear in the same positions, and $C_1$ and $C_2$ are the same.

We now prove the claim. Suppose $a$ is weakly above $b$ in the dual graph of a home configuration. It suffices to show that after performing any slide on the graph, $a$ will remain weakly above  $b$. This is trivial unless the slide moves either $a$ or $b$.

Case 1: $a$ and $b$ are in the same row. Then $a$ must be to the left of $b$. Now, $b$ can only slide if there is a triangle in the graph with $b$ as one vertex, and both other vertices unlabeled. Since the vertex to the left of $b$ is labeled by $a$, the only way $b$ can slide is if both vertices in the row directly below $a$ and $b$ are unlabeled. In this case, $b$ may slide down into either spot in that row.
But this slide moves $b$ into a row below the one containing $a$, so we still have $a$ weakly above $b$. See Figure \ref{samerow}.

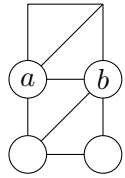
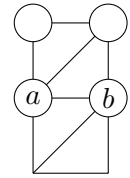
\begin{figure}
\begin{subfigure}[b]{0.4 \textwidth}
\centering
\begin{tikzpicture} 
[dot/.style= {shape=circle, draw, fill = white, minimum size = 0.5cm, inner sep = 0 pt,anchor=center}]
\draw(0,0) grid (1, 2);
\draw(0,0) -- (1,1);
\draw(0,1) -- (1,2);
\node[dot] at (0,1) {$a$};
\node[dot] at (1,1) {$b$};
\node[dot] at (0,0) {};
\node[dot] at (1,0) {};
\end{tikzpicture}
\caption{Tile $b$ cannot slide up, but can slide down if both vertices in the row below are unlabeled.}
\end{subfigure}
\hfill
\begin{subfigure}[b]{0.4 \textwidth}
\centering
\begin{tikzpicture} 
[dot/.style= {shape=circle, draw, fill = white, minimum size = 0.5cm, inner sep = 0 pt,anchor=center}]
\draw(0,0) grid (1, 2);
\draw(0,0) -- (1,1);
\draw(0,1) -- (1,2);
\node[dot] at (0,1) {$a$};
\node[dot] at (1,1) {$b$};
\node[dot] at (0,2) {};
\node[dot] at (1,2) {};
\end{tikzpicture}
\caption{Tile $a$ cannot slide down, but can slide up if both vertices in the row above are unlabeled.}
\end{subfigure}
\caption{Proof of Case 1: When $a$ and $b$ are in the same row.}
\label{samerow}
\end{figure}

Similarly, if we start with $a$ and $b$ in the same row, and $a$ to the left of $b$, then the only case in which $a$ can slide is if both spots in the row above $a$ and $b$ are unlabeled, and $a$ slides upward into one of those  spots. Again, after such a slide, we still have $a$ above $b$.

Case 2: $a$ is in a higher row than $b$. After a slide, each label either remains in the same row or moves to a neighboring row. Hence if suffices to check the case where $a$ is in the row immediately above $b$, and either $a$ slides down into the row containing $b$; or $b$ slides up into the row containing $a$. However, there is only one way in which a tile can slide upward into a row where one tile is labeled already. This can only occur if the labeled tile is on the left, with an unlabeled vertex on the right. Hence if $b$ slides upward into the row containing $a$, then $a$ will be to the left of $b$ after the slide, as desired. Similarly, for $a$ to slide downward into the row containing $b$, the vertex on the left of that row must be unlabeled while the right is labeled by $b$, so again the slide places $a$ to the left of $b$ as desired. See Figure \ref{diffrow}.

\begin{figure}
\begin{subfigure}[b]{0.4 \textwidth}
\centering
\begin{tikzpicture} 
[dot/.style= {shape=circle, draw, fill = white, minimum size = 0.5cm, inner sep = 0 pt,anchor=center}]
\draw(0,0) grid (1, 2);
\draw(0,0) -- (1,1);
\draw(0,1) -- (1,2);
\node[dot] at (0,1) {$a$};
\node[dot] at (1,1) {};
\node[dot] at (0,0) {$b$};
\node[dot] at (1,0) {};
\begin{scope}[xshift = 3 cm]
\draw(0,0) grid (1, 2);
\draw(0,0) -- (1,1);
\draw(0,1) -- (1,2);
\node[dot] at (0,1) {$a$};
\node[dot] at (1,1) {};
\node[dot] at (0,0) {};
\node[dot] at (1,0) {$b$};
\end{scope}
\end{tikzpicture}
\caption{If tile $b$ slides up, it will move into the space to the right of $a$.}
\end{subfigure}
\hfill
\begin{subfigure}[b]{0.4 \textwidth}
\centering
\begin{tikzpicture} 
[dot/.style= {shape=circle, draw, fill = white, minimum size = 0.5cm, inner sep = 0 pt,anchor=center}]
\draw(0,0) grid (1, 2);
\draw(0,0) -- (1,1);
\draw(0,1) -- (1,2);
\node[dot] at (0,2) {$a$};
\node[dot] at (1,2) {};
\node[dot] at (0,1) {};
\node[dot] at (1,1) {$b$};
\begin{scope}[xshift = 3 cm]
\draw(0,0) grid (1, 2);
\draw(0,0) -- (1,1);
\draw(0,1) -- (1,2);
\node[dot] at (0,2) {};
\node[dot] at (1,2) {$a$};
\node[dot] at (0,1) {};
\node[dot] at (1,1) {$b$};
\end{scope}
\end{tikzpicture}
\caption{If tiles $a$ slides down, it will move into the space to the left of $b$.}
\end{subfigure}
\caption{Proof of Case 2: When $a$ and $b$ are in different rows.}
\label{diffrow}
\end{figure}
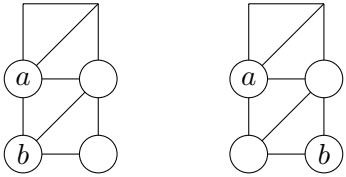
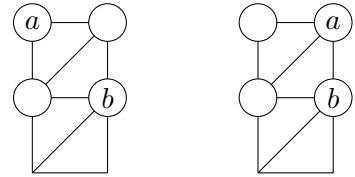

This completes the proof of our claim, and hence of Theorem \ref{thm:parallelogram}.
\end{proof}

We have proved that the board $P_{\varhexagon}(h; 2,m_2)$ with $m_2 \geq 2$ generates the same kind of sliding puzzles as the board $R_{\square}(h; 1,m_2)$ for $m_2 \geq 2$. Before studying the puzzle graph of more general parallelogram boards we first prove a useful lemma that helps us determine when applying a permutation on tiles of a given configuration produces another configuration in the same connected component of the puzzle graph.

\begin{lem}[Conjugation Lemma]
    Let $C$ be a configuration of a board, and let $\sigma$ be a permutation of the tile labels of $C$. To show that $\sigma \cdot C$ is in the same connected component of the puzzle graph as $C$, it suffices to show that there is a configuration $C_1$ in the same connected component as $C$, such that $\sigma \cdot C_1$ is also in the same connected component as $C$.
    \label{conjugation}
\end{lem}

\begin{proof}
    Suppose we have configurations $C_1$, $\sigma \cdot C_1$ as above in the same connected component as $C$. Then we can perform a sequence of slides to transform $C$ into $C_1$, and another sequence of slides to transform $C_1$ into $\sigma \cdot C_1$. By reversing the sequence of slides used to transform $C$ into $C_1$, we may transform $\sigma \cdot C_1$ into $\sigma \cdot C$, so $C$ and $\sigma \cdot C$ are in the same connected component of the puzzle graph.
\end{proof}

We now prove Theorem \ref{thm:parallelogram}, which we restate here:
\begin{thm*}
    For $m_1, m_2 \geq 3$ and $h \geq 3$, 
    \[ puz[P_{\varhexagon}(h; m_1,m_2)]  \]
    is maximally connected, that is, it has one large connected component containing all non-isolated configurations.
\end{thm*}

\begin{proof}
    We prove this by simultaneous induction on $m_1$ and $m_2.$ We check the base case, where $m_1 = m_2 = 3$, computationally. See Section \ref{section:computations}.
    
    Now, suppose the result holds for some $m_1, m_2 \geq 3$. We will show that the result holds for $m_1 + 1, m_2$ and similarly for $m_1, m_2 + 1$. We check the case of $m_1+1,m_2$ here. In this proof, we orient the board so that we have $m_1 + 1$ columns of length $m_2$ running in the vertical direction, with neighboring hexagons in a column sharing a horizontal edge.
    
    Consider any configuration of $puz[P_{\varhexagon}(h; m_1+1,m_2)]$ where at least two holes are adjacent (so the configuration is nonisolated), and let $a,b$ be any two tiles. We claim that we can transpose $a$ and $b$, while leaving the rest of the board fixed. To prove this, we first show that we can perform a sequence of slides that moves $a,b$ and at least $3$ of the holes into some $m_1 \times m_2$ sub-board.
    
    We first slide at least 3 holes into the leftmost column, and call the resulting configuration of the board the \emph{ready configuration}. (Note that we may need to move tiles $a$ and $b$ during this process. We will account for this later.) There are three cases to consider.
    
    Case 1: If $a$ and $b$ are in the leftmost $m_1$ columns in the ready configuration, then $a$, $b$ and three holes are already contained in the sub-board consisting of the first $m_1$ columns, and we are done.
    
    Case 2: If $a$ and $b$ are both in the rightmost column, slide three holes from the leftmost column to the second column from the left. Then tiles $a$ and $b$, and three holes, are in the sub-board consisting of the rightmost $m_1$ columns.
    
    Case 3: The final case is when $a$ is in the rightmost column, but $b$ is not. We may then use the base case, applied to the left $m_1$ columns, perform a sequence of slides that leave three holes and tile $b$ in the part of the board excluding the first and last columns. Again $a$, $b$ and three holes are now in the rightmost $m_1$ columns.
    
    Hence in all three cases, we are able to move $a$, $b$ and three holes into some $m_1 \times m_2$ sub-board of the original board. By inductive assumption, we may then switch tiles $a$ and $b$, without disturbing the rest of the board. By the Conjugation Lemma, it follows that starting with any non-isolated configuration, we may switch tiles $a$ and $b$ while leaving the rest of the board fixed. 
    
    Since any permutation of the tiles (keeping the location of the holes fixed) can be achieved with a sequence of transpositions, this proves that we can achieve any permutation of the tiles given a fixed configuration of the holes. Since it is possible to move holes from any configuration with at least two holes adjacent to any other (ignoring the tiles labels), this implies that all non-isolated configurations of the board are in the same component of the puzzle graph.
\end{proof}

We offer a different proof of Theorem \ref{thm:parallelogram} and more general boards in Section \ref{section:patching} after we introduce a tool that allows us to extend what we know about the puzzle graph of small boards to the puzzle graph of bigger boards obtained by gluing together copies of the smaller boards.
\section{Patching Theorem}
\label{section:patching}

\begin{thm}[Patching Theorem] Let $B$ be a board of any shape with hexagonal tiles. Suppose that $B$ may be written as a union of two smaller boards (``patches'') $B_1$ and $B_2$ such that:
\begin{itemize}
    \item The puzzle graph of the board $B_1(h)$ with $h$ tiles removed has one large connected component containing all non-isolated configurations.
    \item The puzzle graph of the board $B_2(h)$ with $h$ tiles removed has one large connected component containing all non-isolated configurations.
    \item The intersection of $B_1 \cap B_2$ is a connected subset of the board and contains at least $k+1$ hexagons.
\end{itemize}
Then the puzzle graph of $B(h)$ has one connected component containing all non-isolated configurations.
\end{thm}

\begin{proof}
    Let $C$ be a non-isolated configuration of the board $B(h)$, and let $a$ and $b$ be any two tiles. By the Conjugation Lemma, it suffices to show that there is a configuration $C'$ in the same connected component of the puzzle graph as $C$, such that $(a,b) \cdot C'$ is also in the same connected component.
    
    By sliding tiles as needed, we may assume that at least $k$ holes are found in the intersection $B_1 \cap B_2$. If $a$ and $b$ are in the same patch (say, $B_1$), then we may use the connectedness of the puzzle graph of $B_1$ to switch $a$ and $b$, without disturbing the rest of the board.
    
    Suppose without loss of generality that $a$ is in $B_1 \backslash B_2$ while $b$ is in $B_2 \backslash B_1$. Then we may use the connectedness of $B_2$ to move tile $b$ into $B_1 \cap B_2$, without disturbing any tiles of $B_1 \backslash B_2$, and in such a way that there remain at least $h$ holes in $B_1 \cap B_2$. We may then use the connectedness of the puzzle graph of $B_1$ to switch tiles $a$ and $b$, without disturbing the rest of the board. This proves the lemma.
\end{proof}

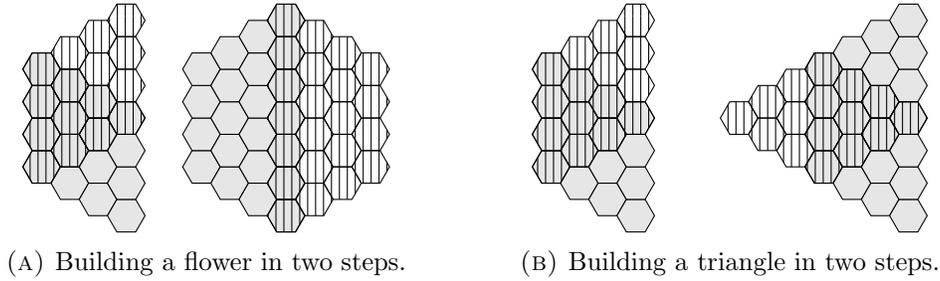
\begin{figure}
    \begin{subfigure}[b]{0.4 \textwidth}
    \centering
    \begin{tikzpicture}[scale=0.5]
    \foreach \j in {0,...,3}{
        \foreach \i in {0,...,3}{
            \node[preaction={fill, gray!20}, smhex] at ({(0.75 * \i},{(\j-\i/2)*sin(60)}) {};}  }
    \foreach \j in {0,...,3}{
        \foreach \i in {0,...,3}{
            \node[smhex, pattern=vertical lines] at ({(0.75 * \i},{(\j+\i/2)*sin(60)}) {};}  } 
    \begin{scope}[xshift=4.25 cm]
    \foreach \i in {0,...,3}{
        \pgfmathsetmacro\last{3+\i} 
        \foreach \j in {0,...,\last}{
            \node[preaction={fill, gray!20}, smhex]  at ({(0.75 * \i},{(\j-\i/2)*sin(60)}) {};}  }
    \foreach \i in {3,...,6}{
        \pgfmathsetmacro\first{\i-4}
        \foreach \j in {\first,...,5}{
            \node[pattern=vertical lines, smhex]  at ({(0.75 *\i},{(\j-\i/2+1)*sin(60)}) {};}  }
    \end{scope}
    \end{tikzpicture}
    \caption{Building a flower in two steps.}
    \end{subfigure}
    \hspace{0.2 in}
    \begin{subfigure}[b]{0.4 \textwidth}
    \centering
    \begin{tikzpicture}[scale=0.5]
    \foreach \j in {0,...,3}{
        \foreach \i in {0,...,3}{
            \node[preaction={fill, gray!20}, smhex] at ({(0.75 * \i},{(\j-\i/2)*sin(60)}) {};}  }
    \foreach \j in {0,...,3}{
        \foreach \i in {0,...,3}{
            \node[smhex, pattern=vertical lines] at ({(0.75 * \i},{(\j+\i/2)*sin(60)}) {};}  } 
    \begin{scope}[xshift=4.25 cm]
    \foreach \i in {0,...,3}{
        \pgfmathsetmacro\xcor{3+0.75*\i}
        \pgfmathsetmacro\last{3+\i} 
        \foreach \j in {0,...,\last}{
            \node[preaction={fill, gray!20}, smhex]  at ({\xcor},{(\j-\i/2)*sin(60)}) {};}  }
    \foreach \i in {1,...,4}{
        \foreach \j in {1,...,\i}{
            \node[pattern=vertical lines, smhex]  at ({(0.75 *\i},{(1+\j-\i/2)*sin(60)}) {};}}
    \foreach \i in {5,...,7}{
        \pgfmathsetmacro\start{\i-3}
        \foreach \j in {\start,...,4}{
            \node[pattern=vertical lines, smhex]  at ({(0.75 *\i},{(1+\j-\i/2)*sin(60)}) {};}}
    \end{scope}
    \end{tikzpicture}
    \caption{Building a triangle in two steps.}
    \end{subfigure}
    \caption{Building boards from parallelogram patches.}
    \label{hexpatch}
\end{figure}

\begin{rmk}
    We note that the last hypothesis in the Patching Lemma can be weakened. As long as the intersection $B_1 \cap B_2$ contains at least two adjacent positions, we can always move at least $k$ holes and one additional tile into $B_1 \cap B_2$, in such a way that we have two adjacent holes--and hence, a non-isolated configuration. We will not need this weaker hypothesis for any of our examples.
\end{rmk}

%% TODO: Add remark about about types of boards, boards that cannot be embedded in plane!!
Corollary \ref{cor:patchedboards}, restated below, is an immediate consequence of the Patching Theorem.
We note that Theorem \ref{thm:connectionnumber} is a consequence of the corollary, together with Theorem \ref{paritytheorem}.
\begin{cor*} 
    Suppose $B$ is either:
    \begin{enumerate}
        \item A flower-shaped board $F_{\varhexagon}(h;m)$ for $m \geq 3$.
        \item A triangular board $T_{\varhexagon}(h;m)$ for $m \geq 5$.
    \end{enumerate}
    Then the puzzle graph for the board $B$ with $3$ or more tiles removed has one large connected component containing all non-isolated configurations. 
\end{cor*}

\begin{proof}
    We know that the puzzle graph of $P_{\varhexagon}(3;m_1,m_2)$ is maximally connected as long as $m_1,m_2 \geq 3$. Our strategy is to build the desired boards out of parallelogram-shaped patches, and apply the Patching Theorem.
    
    For the flower-shaped board, suppose we take a $F_{\varhexagon}(n)$ for $n \geq 3$. Then we may start by overlapping two copies of $P_{\varhexagon}(n)$, so that the overlap is an $n \times n$ equilateral triangle (and hence has more than 4 tiles). This creates a trapezoidal patch. We may then patch together two such trapezoids to form a flower. See Figure \ref{hexpatch}.
    
    For a triangular board, the case is even easier. We begin with two parallelograms patched together as above, then add a third parallelogram patch so that all three patches overlap in central triangle. Again, see Figure \ref{hexpatch}.
\end{proof}

\section{Parity}
\label{section:parity}

We now consider flower, triangle or parallelogram-shaped boards with two tiles removed. Here, parity is a key consideration. In addition, we encounter some obstacles to connectivity of the puzzle graph, which are not present for square puzzles.

\begin{defn}
We say a hexagon is a \emph{tight corner} of a board $B$ if the tile has exactly two neighbors, which form a $2 \times 2$ triangle with the tile in question.
\end{defn}
Hence a parallelogram-shaped board has two tight corners, while a triangular board has three. A flower shaped board has none. Tight corners create some difficulties for boards with only two holes.

\begin{lem}\label{tightcorner}Consider a board with at least one tight corner. Let $C$ be a configuration of the board with exactly two holes, and with tile $a$ is a given corner. Then every configuration $C'$ in the same connected component of the puzzle graph has either a hole or the tile $a$ in the tight corner which is occupied by tile $a$ in $C$
\end{lem}

\begin{proof}
    If $C$ is an isolated configuration, there is nothing to prove. If $C$ is non-isolated, then $C$ has a single pair of adjacent holes, and so does any $C'$ in the same connected component of the puzzle graph. 
    
    Suppose tile $a$ is in its original position. To slide tile $a$ out of the tight corner, we must move the two holes into the hexagons neighboring the corner. We may then slide tile $a$ into either hole, leaving a hole in the corner. Once this happens, however, the only tile capable of sliding is $a$. We may either slide tile $a$ back to its original position, or slide $a$ into the other hole--which leaves a hole in the tight corner.
    
    Hence if $a$ is in its original tight corner, sliding $a$ out of the tight corner leaves a hole in its place. If $a$ is no longer in its corner, all possible slides result in the corner being occupied by a hole or $a$ itself. This proves the lemma.
\end{proof}

\begin{rmk}
We note that a trimmed parallelogram board may be obtained by taking a parallelogram-shaped board, and removing the two tight corners. Also, a trimmed triangle may be obtained by taking a board shaped like an equilateral triangle, and removing the three tight corners. A trimmed parallelogram $P_{\varhexagon}^{tr}(m_1,m_2)$ has no tight corners for $m_1,m_2 \geq 3$, and a trimmed triangle $T_{\varhexagon}^{tr}(m)$ has no tight corners for $m \geq 4$.
\end{rmk}

\begin{lem}\label{corner_count}
    Let $B$ be the board $T_{\hexagon}(m)$ for $m \geq 4$, or the parallelogram-shaped $P_{\varhexagon}(m_1,m_2)$ where $m_1,m_2 \geq 3$. Let $B'$ be the trimmed board obtained by removing the tight corners of $B$, and suppose the puzzle graph of the board $B'$ with two tiles removed has $c$ connected components containing non-isolated configurations.
    
    Then the number of connected components of the puzzle graph of $B$ with two tiles removed, containing non-isolated configurations, is 
        \begin{itemize}
            \item $\displaystyle 2c  \binom{m_1m_2-2}{2}$ for a parallelogram,
            \item $\displaystyle 6c \binom{m(m+1)/2-2}{3}$ for a triangle.
        \end{itemize}
\end{lem}

\begin{proof}
    Let $C_1$ be a non-isolated configurations of the board $B$ with no holes in tight corners, and let $C_2$ be another such configuration in the same component of the puzzle graph. Let $C_1'$ be the configuration of $B'$ induced by $C_1$. As in the proof of Lemma \ref{tightcorner}, we may slide a tile out of a tight corner of the board $B$; however, this tile must be returned to its original position, without disturbing the rest of the board, before any slides involving other tiles can take place. Hence we may assume  that the sequence of slides transforming $C_1$ into $C_2$ involves no slides into or out of tight corners. The possibilities for $C_2$ are then in bijection with the non-isolated configurations of $B'$ with two holes removed, which are in the same connected component of the puzzle graph of $B'(2)$ as $C_1'.$ 
    
    Let $t$ be the total number of tiles on the board $B$. It follows that configurations with no holes in their tight corners may be partitioned into $\displaystyle c \cdot k! \binom{t}{k}$ connected components, where $k$ is the number of tight corners of $B$. We have $\displaystyle k! \binom{t}{k}$ options for the placement of tiles in the tight corners, and we obtain $c$ distinct connected components for each choice of tight corner tiles. Since every non-isolated configuration is in the same connected component as a configuration with no holes in tight corners, we have in fact given a complete count of connected components containing non-isolated configurations, and the lemma is proved.
\end{proof}

 Having addressed the problem of tiles getting stuck in tight corners, we now turn to a more fundamental obstacle to connectivity of the puzzle graph: parity of permutations on tile labels. Here we prove Theorem \ref{paritytheorem} that states that parallelogram boards with two holes have the weak parity property.
\begin{thm*}
 The board $P_{\varhexagon}(2; m_1,m_2)$ has the weak parity property for $m_1, m_2 \geq 3.$
%    Let $C$ and $C'$ be two non-isolated configurations in the same connected component of $hex(m_1m_2-2:m_1, m_2)$, where the two holes are located in the same position in $C$ and $C'$. Then the permutation of the tiles necessary to transform $C$ into $C'$ is even.
\end{thm*}

\begin{proof} We define an \textbf{augmented configuration} of $P_{\varhexagon}(2; m_1,m_2)$ to be a configuration of $P_{\varhexagon}(2; m_1,m_2)$ where the two holes have been labeled $0$ and $-1$. The rules for sliding tiles remain the same as when the holes were unlabeled. However, labeling allows us to keep track of the relative position of the two holes, as we slide the tiles.

We pass from the board $P_{\varhexagon}(2; m_1,m_2)$ to the corresponding dual graph, which with our conventions is a grid with diagonal edges joining the upper-right and lower-left corners of each square cell. See Figure \ref{2bym} for an small example. Let $\bar{C}$ be an augmented configuration obtained by labeling the holes of $C$ with $-1$ and $0$. 

For any augmented configuration $\bar{C'}$ whose underlying configuration is in the same component of $puz[P_{\varhexagon}(2; m_1,m_2)]$ as $C$, we define the \emph{augmented parity} of $\bar{C'}$ as follows. We ignore the diagonals of the dual graph, and consider only the vertical and horizontal edges. Let $p_1$ be the parity (zero for odd, one for even) of the taxicab distance from the hole labeled $-1$ to the upper-left vertex of the grid. Let $p_2$ be the parity of the taxicab distance of the hole labeled $0$ to the upper-left vertex. Finally, let $p_3$ be the parity of the permutation of \emph{all} positions on the board (both tiles and labeled holes) needed to transform $\bar{C}$ into $\bar{C}'$. We define the augmented parity $\bar{C}'$ as the parity of the sum $p_1+p_2+p_3$.

We now claim that that we can transform $C$ into $C'$ using a sequence of moves which preserve the augmented parity.

First, consider the case where we slide a tile vertically or horizontally. A single slide transposes a hole and a tile, which changes the parity $p_3$ of the overall permutation of positions on the board. In addition, with each vertical or horizontal slide, exactly one of the holes moves a distance of one in the horizontal or vertical direction, while the other hole does not move at all. So exactly one of $p_1$ or $p_2$ changes. Hence, each vertical or horizontal slide changes exactly two of the $p_i$'s: $p_3$ always changes, and exactly one of $p_1$ or $p_2$ changes. It follows that such a slide does not change the augmented parity of the augmented configuration.

It remains to deal with diagonal slides. Note that when a tile slides diagonally, the parity of its taxicab distance from a given point will not change, as a diagonal slide is equivalent to moving one position in the horizontal direction and one position in the vertical direction. Hence a diagonal slide changes the augmented parity because $p_1$ and $p_2$ remain unchanged while $p_3$ changes. 

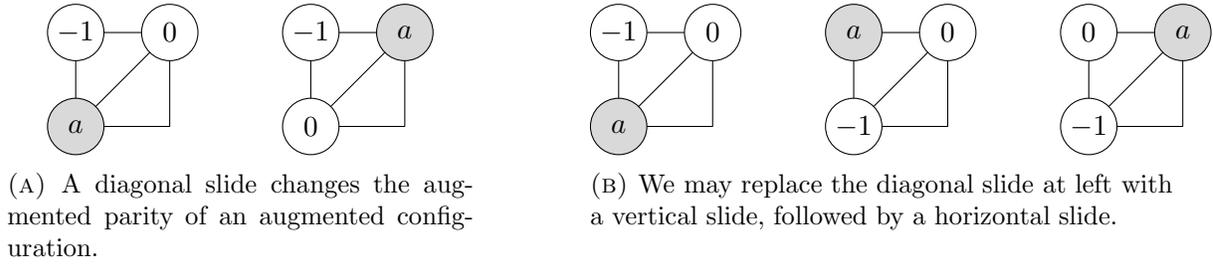
\begin{figure}
\centering
\begin{subfigure}[t]{0.4 \textwidth}
\centering
\begin{tikzpicture}
[scale=1.25, dot/.style= {shape=circle, draw, fill = white, minimum size = 0.75cm, inner sep = 0 pt,anchor=center},
gdot/.style= {shape=circle, draw, fill = gray!30, minimum size = 0.75cm, inner sep = 0 pt,anchor=center}]
\draw(0,0) grid (1, 1);
\draw(0,0) -- (1,1);
\node[gdot] at (0,0) {$a$};
\node[dot] at (0,1) {$-1$};
\node[dot] at (1,1) {$0$};
\node at (1,0) {};
\begin{scope}[xshift=2.5cm]
\draw(0,0) grid (1, 1);
\draw(0,0) -- (1,1);
\node[dot] at (0,0) {$0$};
\node[dot] at (0,1) {$-1$};
\node[gdot] at (1,1) {$a$};
\node at (1,0) {};
\end{scope}
\end{tikzpicture}
\caption{A diagonal slide changes the augmented parity of an augmented configuration.}
\end{subfigure}
\hfill
\begin{subfigure}[t]{0.5 \textwidth}
\centering
\begin{tikzpicture} 
[scale=1.25, dot/.style= {shape=circle, draw, fill = white, minimum size = 0.75cm, inner sep = 0 pt,anchor=center},
gdot/.style= {shape=circle, draw, fill = gray!30, minimum size = 0.75cm, inner sep = 0 pt,anchor=center}]
\draw(0,0) grid (1, 1);
\draw(0,0) -- (1,1);
\node[gdot] at (0,0) {$a$};
\node[dot] at (0,1) {$-1$};
\node[dot] at (1,1) {$0$};
\node at (1,0) {};
\begin{scope}[xshift=2.5cm]
\draw(0,0) grid (1, 1);
\draw(0,0) -- (1,1);
\node[dot] at (0,0) {$-1$};
\node[gdot] at (0,1) {$a$};
\node[dot] at (1,1) {$0$};
\node at (1,0) {};
\end{scope}
\begin{scope}[xshift=5cm]
\draw(0,0) grid (1, 1);
\draw(0,0) -- (1,1);
\node[dot] at (0,0) {$-1$};
\node[dot] at (0,1) {$0$};
\node[gdot] at (1,1) {$a$};
\node at (1,0) {};
\end{scope}
\end{tikzpicture}
\caption{We may replace the diagonal slide at left with a vertical slide, followed by a horizontal slide.}
\end{subfigure}
\caption{Dealing with diagonal slides in the proof of \ref{paritytheorem}.}
\label{fig:diagonalslide}
\end{figure}

Fortunately, we claim that no diagonal slides are required to reach $\bar{C}'.$ Indeed, suppose we want to slide tile $a$ up and to the right. Then tile $a$ must have either a hole immediately above it, or a hole immediately to the right. In the first case, we can achieve the same effect on the underlying configuration by sliding tile $a$ first vertically up, then horizontally to the right. See Figure \ref{fig:diagonalslide}. In the second case, we can slide tile $a$ first to the right, then vertically. The case of sliding a tile down and to the left is similar.

Hence any diagonal slide may be replaced with a sequence of vertical and horizontal slides, and the proof is complete.
\end{proof}

We can now prove Corollary \ref{cor:weakparity} that we restate below. 

\begin{cor*}
    Any board with exactly two holes that is a sub-board of a parallelogram board has the weak parity property. Hence boards of any shape that we explore in this paper with exactly two holes have the weak parity property.
\end{cor*}

\begin{proof}
    Let $B$ be a board which can be embedded in a larger parallelogram-shaped board $P$. Suppose we can apply a sequence of slides to the board $B(2)$ of shape $B$ with exactly two holes, and achieve an odd permutation of the tiles. Viewing $B$ as a sub-board of $P$, we have a sequence of slides which produces an odd permutation of the tiles of $P(2)$, contradicting Theorem \ref{paritytheorem}.
\end{proof}

Now, using also Corollary \ref{cor:patchedboards}, we can finish the proof of Theorem \ref{thm:connectionnumber}.
\begin{proof}
By Theorem \ref{thm:parallelogram} and Corollary \ref{cor:patchedboards}, with three or more holes the puzzle graphs of large-enough parallelogram, triangle or flower-shaped boards are maximally connected. By Theorem \ref{paritytheorem} with only two holes the puzzle graphs of these boards are not maximally connected.
\end{proof}
%\begin{defn}
%    We say a board $B$ has the \emph{parity property} if the puzzle graph of $B$ with two holes removed has exactly two components containing all non-isolated configurations. Two configurations are in the same connected component of the puzzle graph if one can be obtained from other other by first sliding the holes to their desired location, then applying an even permutation to the tiles.
%\end{defn}

We now give a variation on the Patching Theorem for boards with exactly two holes, which allows us to use patching to prove the strong parity property.

\begin{thm}(Strong Parity Property via Patching)
    Suppose $B$ is a board which can be written as a union of two smaller boards, or ``patches'', $B_1$ and $B_2,$ such that all of the following hold:
    \begin{enumerate}
        \item $B_1$ and $B_2$ each contain at least five hexagons.
        \item $B_1 \cap B_2$ contains at least four hexagons, at least two of which are adjacent.
        \item When we consider $B_1$ and $B_2$ separately, each one has, with exactly two holes, the strong parity property.
    \end{enumerate}
    Then if $B$ with exactly two holes has the weak parity, it has the strong parity property as well.
    \label{paritypatching}
\end{thm}

\begin{proof}
  Recall that the alternating group of even permutations is generated by three-cycles $(a,b,c)$. Let $C$ be a non-isolated configuration of the board $B$, and let $a,b,c$ be any three tiles. By the Conjugation Lemma, it suffices to show that there is a configuration $C'$ in the same connected component of the puzzle graph as $C$, such that $(a,b,c) \cdot C'$ is also in the same connected component.
    
    By sliding tiles as needed, we may assume both holes are in the intersection $B_1 \cap B_2$. If $a,b$ and $c$ are in the same patch (say, $B_1$), then we may use the assumptions about the puzzle graph of $B_1$ to apply the three-cycle $(a,b,c)$, without disturbing the rest of the board.
    
    Suppose without loss of generality that $a,b \in B_1$ while $c$ is in $B_2 \backslash B_1$. We may assume further that $a \in B_1 \backslash B_2$, or else we would be in the first case. Note, however, that we may have $b \in B_1 \cap B_2$. We use the assumption on $B_2$ to move tile $c$ into $B_1 \cap B_2$, without disturbing any tiles of $B_1 \backslash B_2$, and in such a way that the two holes remain in $B_1 \cap B_2$, and $b$ remains in $B_1 \cap B_2$ if needed. (Note that we can do this with an even permutation of tiles, since $B_2$ has at least 3 labeled tiles.) We may then use the parity property of $B_1$ to apply the permutation $(a,b,c)$. This proves the lemma.
\end{proof}

We provide now the proof of Theorem \ref{thm:strongparity}. We begin by stating here again Theorem \ref{thm:strongparity}.
\begin{thm*}
    Suppose $B$ is a hexagonal board with any of the following shapes:
    \begin{itemize}
        \item A flower $F_{\varhexagon}(m)$ where $m \geq 3.$
        \item A trimmed parallelogram $P_{\varhexagon}^{tr}(m_1,m_2)$ where $m_1,m_2 \geq 3$ and 
        \[\max\{m_1,m_2\} \geq 4.\]
        \item A trimmed triangle $T^{tr}_{\varhexagon}(m)$ for $m \geq 5$.
    \end{itemize}
    Then $B$ with two holes has the strong parity property.
\end{thm*}

\begin{proof}
    %By Theorem \ref{paritytheorem}, it is enough to show that for each board, if one non-isolated configuration can be obtained from another by sliding holes and then applying an even permutation of the tiles, then those two configurations are in the same connected component.
    
    We begin with the case $P_{\varhexagon}^{tr}(3,4)$. We check this computationally---see discussion in Section \ref{section:computations}. It follows by induction that the result holds for all larger trimmed parallelograms, as we may cover any such board with two smaller trimmed parallelogram patches, whose intersection covers either all but the first and last row of the board, or all but the first and last column, and then apply Theorem \ref{paritypatching}.
    
    The flower $F_{\varhexagon}(m)$ may be constructed by joining two copies of $P_{\varhexagon}^{tr}(m+1,m+1)$ as shown in Figure \ref{evenpatch}, and then joining two of the resulting patches to form a flower. See Figure \ref{evenpatch}. Note that this is analogous to the way we built up flower boards for parallelogram patches in Figure \ref{hexpatch}.
    
    For a trimmed triangle, the smallest case can easily be constructed from two trimmed $3 \times 4$ parallelogram patches. Larger trimmed triangles may be constructed by successively joining three trimmed parallelogram patches, as in Figure \ref{evenpatch}.
\end{proof}

\begin{figure}
    \begin{subfigure}[b]{0.35 \textwidth}
    \centering
    \begin{tikzpicture}[scale=0.5]
    \foreach \i in {0,...,3}{
         \node[preaction={fill, gray!20}, smhex] at ({(0.75 * \i},{(-\i/2)*sin(60)}) {};}  
    \foreach \j in {1,...,3}{
        \foreach \i in {0,...,4}{
            \node[preaction={fill, gray!20}, smhex] at ({(0.75 * \i},{(\j-\i/2)*sin(60)}) {};}  }
    \foreach \i in {1,...,4}{
         \node[preaction={fill, gray!20}, smhex] at ({(0.75 * \i},{(4-\i/2)*sin(60)}) {};} 
    \foreach \j in {0,...,3}{
        \node[smhex, pattern=vertical lines] at (0,{(\j)*sin(60)}) {};
    }
    \foreach \j in {-1,...,3}{
        \foreach \i in {1,...,3}{
            \node[smhex, pattern=vertical lines] at ({(0.75 * \i},{(\j+\i/2)*sin(60)}) {};}  } 
    \foreach \j in {0,...,3}{
        \node[smhex, pattern=vertical lines] at (3,{(1+\j)*sin(60)}) {};
    }
    \begin{scope}[xshift=5 cm]
    \foreach \i in {0,...,3}{
        \pgfmathsetmacro\last{3+\i} 
        \foreach \j in {0,...,\last}{
            \node[preaction={fill, gray!20}, smhex]  at ({(0.75 * \i},{(\j-\i/2)*sin(60)}) {};}  }
    \foreach \j in {0,...,5}{
        \node[preaction={fill, gray!20}, smhex] at (3,{(\j-1)*sin(60)}) {};
    }
    \foreach \j in {0,...,5}{
        \node[pattern=vertical lines, smhex] at (1.5,{(\j-1)*sin(60)}) {};
    }
    \foreach \i in {3,...,6}{
        \pgfmathsetmacro\first{\i-4}
        \foreach \j in {\first,...,5}{
            \node[pattern=vertical lines, smhex]  at ({(0.75 *\i},{(\j-\i/2+1)*sin(60)}) {};}  }
    \end{scope}
    \end{tikzpicture}
    \caption{Building a flower in two steps.}
    \end{subfigure}
    \hspace{0.1 in}
    \begin{subfigure}[b]{0.2 \textwidth}
    \centering
    \begin{tikzpicture}[scale=0.5]
    \foreach \j in {0,...,2}{
        \foreach \i in {1,...,2}{
            \node[preaction={fill, gray!20}, smhex] at ({(0.75 * \i},{(\j-\i/2)*sin(60)}) {};}  }
    \foreach \j in {1,...,2}{
        \node[preaction={fill, gray!20}, smhex] at (0,{((\j-1)*sin(60)}) {};
    }
    \foreach \j in {1,...,2}{
        \node[preaction={fill, gray!20}, smhex] at (2.25,{((\j-1.5)*sin(60)}) {};
    }
    \foreach \j in {0,...,1}{
        \node[pattern=vertical lines, smhex] at (0,{(\j*sin(60)}) {};
    }
    \foreach \j in {0,...,1}{
        \node[pattern=vertical lines, smhex] at (2.25,{(\j+0.5)*sin(60)}) {};
    }
    \foreach \j in {-1,...,1}{
        \foreach \i in {1,...,2}{
            \node[smhex, pattern=vertical lines] at ({(0.75 * \i},{(\j+\i/2)*sin(60)}) {};}  } 
    \end{tikzpicture}
    \caption{Building $T_{\varhexagon}^{tr}(5)$.}
    \end{subfigure}
    \hspace{0.1 in}
    \begin{subfigure}[b]{0.35 \textwidth}
    \centering
    \begin{tikzpicture}[scale=0.5]
    \foreach \j in {0,...,3}{
        \foreach \i in {1,...,2}{
            \node[preaction={fill, gray!20}, smhex] at ({(0.75 * \i},{(\j-\i/2)*sin(60)}) {};}  }
    \foreach \j in {1,...,3}{
        \node[preaction={fill, gray!20}, smhex] at (0,{((\j-1)*sin(60)}) {};
    }
    \foreach \j in {1,...,3}{
        \node[preaction={fill, gray!20}, smhex] at (2.25,{((\j-1.5)*sin(60)}) {};
    }
    \foreach \j in {1,...,3}{
        \node[pattern=vertical lines, smhex] at (0,{(\j*sin(60)}) {};
    }
    \foreach \j in {1,...,3}{
        \node[pattern=vertical lines, smhex] at (2.25,{(\j+0.5)*sin(60)}) {};
    }
    \foreach \j in {0,...,3}{
        \foreach \i in {1,...,2}{
            \node[smhex, pattern=vertical lines] at ({(0.75 * \i},{(\j+\i/2)*sin(60)}) {};}  } 
    \begin{scope}[xshift=2.75 cm]
    \foreach \i in {0,...,2}{
        \pgfmathsetmacro\xcor{3+0.75*\i}
        \pgfmathsetmacro\last{3+\i} 
        \foreach \j in {0,...,\last}{
            \node[preaction={fill, gray!20}, smhex]  at ({\xcor},{(\j-\i/2)*sin(60)}) {};}  }
    \foreach \j in {0,...,4}{
        \node[preaction={fill, gray!20}, smhex] at
        (5.25,{(\j-0.5)*sin(60)}) {};
    }
    \foreach \i in {2,...,4}{
        \foreach \j in {1,...,\i}{
            \node[pattern=vertical lines, smhex]  at ({(0.75 *\i},{(1+\j-\i/2)*sin(60)}) {};}}
    \foreach \i in {5,...,6}{
        \pgfmathsetmacro\start{\i-3}
        \foreach \j in {\start,...,4}{
           \node[pattern=vertical lines, smhex]  at ({(0.75 *\i},{(1+\j-\i/2)*sin(60)}) {};}}
        \end{scope}
    \end{tikzpicture}
    \caption{Building $T_{\varhexagon}^{tr}(m)$ for $m>5$.}
    \end{subfigure}
    \caption{Building boards from trimmed parallelogram patches.}
    \label{evenpatch}
\end{figure}
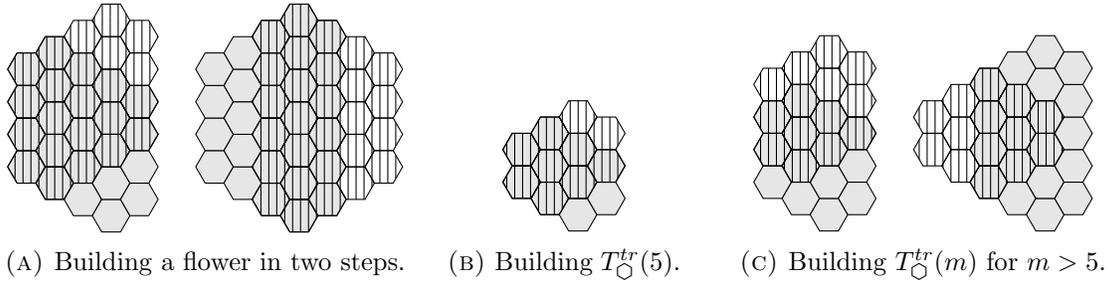

Can now easily prove Theorem \ref{thm:countcomponents}, which we restate below.
\begin{thm*}
    If the board $B$ is
    \begin{enumerate}
        \item A parallelogram $P_{\varhexagon}(m_1,m_2)$, where $m_1,m_2\geq 3$ and $\max\{m_1,m_2\} \geq 4.$
        \item A triangle
        $T_{\varhexagon}(m)$ where $m \geq 5$.
    \end{enumerate}
    Then the number of connected components of $puz[B(2)]$ containing non-isolated configurations is given by
    \begin{itemize}
        \item $\displaystyle 4  \binom{m_1m_2-2}{2}$ for a parallelogram,
        \item $\displaystyle 12 \binom{m(m+1)/2-2}{3}$ for a triangle.
    \end{itemize}
\end{thm*}

\begin{proof}
    Applying Theorem \ref{thm:strongparity}, we can use $c = 2$ in Lemma \ref{corner_count}.
\end{proof}

\section{Computational results for smaller boards}
\label{section:computations}

We now check computationally the smaller triangular, parallelogram, trimmed parallelogram, and flower-shaped boards. We note that it is not necessary to compute the entire puzzle graph to determine the number of connected components. Indeed, computing a single connected component is enough, as we explain below. Python code for computing a list of all configurations in a sample connected component, as well as the complete output files from our computations, can be found at our Github repository \cite{hexagonSoftware}.

\begin{prop}
Let $B(h)$ be a board with $h$ holes, and let $S$, $S'$ be two connected components of $puz[B(h)]$. Then there is a graph isomorphism from $S$ to $S'$ which preserves the locations of the holes in a configuration.
\end{prop}

\begin{proof}
Consider two configurations $C$ and $C'$ of $B(h)$ which have holes in precisely the same positions, but which do not lie in the same component of $puz[B(h)]$. Note that  permuting the tile labels commutes with sliding tiles on the board. Hence the permutation of labels which transforms $C$ into $C'$ induces a graph isomorphism between the connected component of $puz[B(h)]$ containing $C$ and the component containing $C'$. Moreover, this isomorphism preserves the positions of the holes in a configuration. 
\end{proof}

\begin{cor}\label{counting_trick}
Let $t$ be the number of tiles of $B(h)$. The number of connected components of $puz[B(h)]$ containing non-isolated configurations is $t!$ divided by the number of configurations in any such connected component of $puz[B(h)]$ which have their holes in a specified set of positions.
\end{cor}

\begin{proof}
Suppose we choose a set of positions for the holes of $B(h)$ which yields a non-isolated configuration, and call any configuration of $B(h)$ with holes in the specified positions a \emph{home configuration}. Let $t$ be the number of tiles on the board $B(h)$. The total number of home configurations of $B(h)$ is hence $t!$, the number of possible permutations of the tile labels. By the argument in the preceding paragraph, each connected component $puz[B(h)]$ containing non-isolated configurations contains the same number of these home configurations. Thus the number of such connected components is $t!$ divided by the number of home configurations found in a given component.
\end{proof}

\subsection{The board $F_{\varhexagon}(2)$}

We first give our computational results for the flower-shaped board $F_{\varhexagon}(2)$, which has six hexagonal tiles arranged in a ring around a central tile. We include here the total number of configurations in a connected component of the puzzle graph, to convey the difficulty level of the corresponding sliding puzzles.

\begin{table}[h!]
\begin{tabular}{|c|c|c|}
\hline
Holes & Components & Component Size\\ \hline
2 & 24 & 60 \\ \hline
3 & 6 & 132 \\ \hline
4 & 1 & 210\\ \hline
\end{tabular}
\caption{Results for $F_{\varhexagon}(2)$.} 
\end{table}

\subsection{Small triangular boards}

We now give results for triangular board $T_{\varhexagon}(h; m)$, where $m=2,3,4$. For each $m$, we check values of $h$ starting at $2$, until we reach the connected number of $Tr_{\varhexagon}(m)$. Again, components here refer to components of the puzzle graph containing non-isolated configurations.

\begin{table}[h!]
\begin{tabular}{|c|c|c|c|}
\hline
 $m$ & Holes & Components & Component Size\\ \hline
2 & 2 & 1 & 3 \\ \hline
3 & 2 & 24 & 9\\ \hline
3 & 3 & 6 & 19\\ \hline
3 & 4 & 1 & 30\\ \hline
4 & 2 & 8064 &  90\\ \hline
4 & 3 & 1 & 498960 \\ \hline
\end{tabular}
\caption{Results for $T_{\varhexagon}(m)$.}
\end{table}

The results for $m=2$ can be found easily by inspection. For $m=3$, we use Python to compute a single component, and apply Corollary \ref{counting_trick}. We use this approach as well for the case $m=4, h=3$. 

For the case $m=4$, $h=2$, we instead applied Lemma \ref{corner_count}, together with our results on  $F_{\varhexagon}(2; 2),$ to give the number of connected components. To obtain the number of configurations in a component, note that for $F_{\varhexagon}(2; 2)$ we have 60 configurations per component, with 5 configurations for each possible position of the two adjacent holes. This gives 60 configurations of the triangular board, where no holes are found in the tight corners. For each corner, there are 5 possible configurations of the board where the two holes are adjacent to the given corner, and we may then slide a tile from the tight corner into either of the holes. Hence we have 10 configurations with a hole in the given corner, giving 30 more configurations, for a total of 90. This counting method was less time-consuming than running the needed computations on the author's laptop.

\subsection{Parallelograms, trimmed parallelograms, and parity}

Computationally, we find that $P_{\varhexagon}(3; 3,3)$ has a puzzle graph that is maximally connected, while $P^{tr}_{\varhexagon}(2;3,4)$ has a puzzle graph with exactly two connected components containing non-isolated configurations.  By the proposition below, it follows that $P^{tr}_{\varhexagon}(2;3,4)$ has the strong parity property.

\begin{prop}
If $B(h)$ has the weak parity property, and $puz[B(h)]$ has exactly two connected components containing non-isolated configurations, then $B(h)$ has the strong parity property. 
\end{prop}

\begin{proof}
Note that the number of even permutations of $t$ tiles is $t!/2$. Hence for each of the two components of $puz[B(h)]$ containing non-isolated configurations, the number of configurations with holes in a specified set of positions is equal to the number of even (respectively odd) permutations of the tile labels. It follows that all possible even permutations of the tile labels must be found in one large connected component, while all possible odd permutations must be found in the other.
\end{proof}

\subsection{God's number for small boards}

In addition to finding the number of configurations in a sample connected component of the puzzle graph, our Python code allows us to find bounds on God's number for small puzzles. Our code uses a breadth-first search (BFS) algorithm. We begin with a starting configuration, and iterate as follows. At the $n^{th}$ iteration, we start with a list of configurations at distance less than or equal to $n$ from the starting configuration. We then find any neighbors of the configurations at distance $n$ from the start, which are not already on the list. These are precisely the configurations which are at distance $n+1$ from the original configuration. We add these new configurations to our list, and repeat the process. The algorithm terminates when it no finds any new configurations, meaning that we have reached the farthest point in the puzzle graph from our original configuration. Hence the algorithm computes a spanning tree for the component of $puz[B(h)]$ containing our starting configuration, with the starting configuration as a root.

We note that God's number is the minimum depth of a tree computed using the BFS algorithm, over all possible starting configurations. By symmetry, it is enough to check configurations with all possible starting positions of the holes, as the trees for two configurations with holes in the same starting locations will be isomorphic.

\begin{prop}
    Let $d$ be the depth of the tree computed by our BFS algorithm
    when building a component of the puzzle graph $puz[B(h)]$, starting from configuration $C$. Then God's number for the board $B(h)$ is bounded below by $d$, and above by $2d$
\end{prop}

\begin{proof}
    Since $d$ is the distance from $C$ to some configuration $C'$, $d$ is certainly a lower bound on God's number. 
    
    We can find a path between any two configurations $D$, $D'$ by first finding a path from $D$ to $C$ in the spanning tree, and then a path from $C$ to $C'$. Since every configuration is a distance of at most $d$ from $C$, concatenating these two paths gives a path of length at most $2d$ from $D$ to $D'$. 
    The result follows, since all components of $puzz[B(h)]$ containing non-isolated configurations are isomorphic.
\end{proof}

Using our BFS algorithm, we found the following bounds for God's number, using the arguments given above. We conjecture that in many cases, the depth of a single tree is in fact God's number for the board overall, but have not verified this computationally. 

\begin{table}
\begin{tabular}{|c|c|c|c|}
    \hline
    Board & Holes & Lower Bound & Upper Bound \\ \hline
    $F_{\varhexagon}(3)$ & 2 & 16 & 32\\ \hline
    $F_{\varhexagon}(3)$ & 3 & 12 & 24\\ \hline
    $F_{\varhexagon}(3)$ & 4 & 8 & 16\\ \hline
    $T_{\varhexagon}(3)$ & 2 & 3& 6\\ \hline 
    $T_{\varhexagon}(3)$ & 3 & 4 & 8\\ \hline
    $T_{\varhexagon}(3)$ & 4 & 6 & 12\\ \hline
    $T_{\varhexagon}(4)$ & 2 &17 & 34\\ \hline
    $T_{\varhexagon}(4)$ & 3 & 56 & 112\\ \hline
    $P_{\varhexagon}(3,3)$ & 3 & 51 & 102\\ \hline
    $P_{\varhexagon}^{tr}(3,4)$ & 2 & 78 & 156\\ 
    \hline
\end{tabular}
\caption{Bounds on God's number for some small boards.}
\end{table}

\section{Final Remarks and Open problems}

In our GitHub repository \cite{hexagonSoftware} we provide 3D models for the small hexagonal boards that we have studied in this paper. These have been kindly developed by Henry Segerman and can be use for 3D printing and playing with some of the hexagonal sliding puzzles that we have studied here. 

Recently \cite{alpert2020discrete}, the asymptotic growth of God's number of $puz[R_{\square}(h;m)]$ with $h \geq 2$, and $puzz[P_{\varhexagon}(h; m)]$ with $h \geq 6$ and $m=k^2$ for $k\geq1$ has been determined up to a constant factor. In these cases, the puzzle graph is maximally connected. It remains open to find God's number for the family of boards that we have studied here. Here, we have given bounds for the God's numbers for some small hexagonal boards. It is possible to use the algorithms that we have implemented to find the exact values of the God's number for these sliding puzzles, by performing an exhausting search.

For generating the puzzle graph, labeled tiles are only allowed to be on positions determined by the vertices of the dual graph of the underlying tessellation. A higher dimensional topological space that contains the puzzle graph is generated by allowing the tiles to move continuously on the board as long as two tiles do not overlap. This topological space has been defined and studied for square sliding puzzles on boards with a rectangular shape; it is called the configuration space of hard squares on a rectangle \cite{alpert2021configuration}. Recently \cite{alpert2021configuration}, it has been proven that the configuration space of the 15 Puzzle deformation retracts to the puzzle graph of  $R_{\square}(1;4,4)$. More generally, in this paper it was proven that the configuration space of rectangular shaped boards $R_{\square}(1;m_1,m_2)$ deformation retracts to a one-dimensional subspace homeomorphic to $puz[R_{\square}(1;m_1,m_2)]$. This relationship between the puzzle graph and the configuration space of $R_{\square}(h;m_1,m_2)$ holds no longer true for $h\geq 2$ that is when maximal connectivity happens. In the case of the hexagonal sliding puzzles that we have studied here, it could be the case that before maximal connectivity the same topological relationship will hold between the puzzle graph and the configuration space.

\section{Acknowledgments}
We thank Hannah Alpert and Ulrich Bauer for useful conversations about configuration spaces of sliding puzzles. This project received funding from the European Union’s Horizon 2020 research and innovation program under the Marie Skłodowska-Curie grant agreement No. 754462.

\bibliography{hexagons}
\bibliographystyle{plain}

\end{document}